\documentclass[oneside,onecolumn,12pt]{article}

\usepackage{blindtext} 

\usepackage[sc]{mathpazo} 
\usepackage[T1]{fontenc} 
\usepackage{microtype} 
\usepackage{amsmath}
\usepackage{amsthm}
\usepackage{amsfonts}

\usepackage{mathrsfs}
\usepackage{txfonts}

\usepackage{palatino,amssymb,epsfig,latexsym}

\usepackage{verbatim}

\usepackage{graphicx}
\usepackage{graphics}

\usepackage{tikz}
\usepackage[english]{babel} 

\usepackage[hmarginratio=1:1,top=32mm,columnsep=20pt]{geometry} 
\usepackage[hang, small,labelfont=bf,up,textfont=it,up]{caption} 
\usepackage{booktabs} 

\usepackage{lettrine} 

\usepackage{enumitem} 
\setlist[itemize]{noitemsep} 

\usepackage{abstract} 

\usepackage{titlesec} 
\titleformat{\section}[block]{\large\scshape\centering}{\thesection.}{1em}{} 
\titleformat{\subsection}[block]{\large}{\thesubsection.}{1em}{} 

\usepackage{titling} 

\usetikzlibrary{patterns}  
\usetikzlibrary{calc}
\usetikzlibrary{through}
\usetikzlibrary{chains}
\usetikzlibrary{decorations.pathreplacing,decorations.markings}
 \tikzset{
  on each segment/.style={
    decorate,
    decoration={
      show path construction,
      moveto code={},
      lineto code={
        \path [#1]
        (\tikzinputsegmentfirst) -- (\tikzinputsegmentlast);
      },
      curveto code={
        \path [#1] (\tikzinputsegmentfirst)
        .. controls
        (\tikzinputsegmentsupporta) and (\tikzinputsegmentsupportb)
        ..
        (\tikzinputsegmentlast);
      },
      closepath code={
        \path [#1]
        (\tikzinputsegmentfirst) -- (\tikzinputsegmentlast);
      },
    },
  },
  mid arrow/.style={postaction={decorate,decoration={
        markings,
        mark=at position .5 with {\arrow[#1]{stealth}}
      }}},
}
\usetikzlibrary{arrows}

\newtheorem{theorem}{Theorem}[section]
\newtheorem{lemma}[theorem]{Lemma}

\newtheorem{proposition}[theorem]{Proposition}
\newtheorem{corollary}[theorem]{Corollary}

\newtheorem{remark}[theorem]{Remark}
\makeatletter

\newcommand{\Rmnum}[1]{\expandafter\@slowromancap\romannumeral #1@}
\makeatother

\numberwithin{equation}{section}

\pretitle{\begin{center}\Huge\bfseries} 
\posttitle{\end{center}} 
\title{Polygons inscribed in Jordan curves with prescribed edge ratios} 
\author{%
\textsc{Yaping Xu}\\[1ex]
\normalsize School of Mathematics, Hunan University, \\
\normalsize Changsha 410082, China; \\ 
\normalsize {email: xuyaping@hnu.edu.cn} 
\and 
\textsc{Ze Zhou} \\[1ex] 
\normalsize College of Mathematics and Statistics, Shenzhen University, \\
\normalsize Shenzhen 518060, China; \\ 
\normalsize{email: zhouze@szu.edu.cn} 
}
\date{} 
\renewcommand{\maketitlehookd}{%
\begin{abstract}

Let $J$ be a simple closed curve in $\mathbb R^{k}$ $(k\geq2)$ that is differentiable with non-zero derivative at a point $A_0\in J$. For a tuple of positive reals $a_1,\cdots,a_n$ $(n\geq3)$, each of which is less than the sum of the others, we show that there exists a polygon $Q_n$ inscribed in $J$ with sides of lengths proportional to $(a_1,\cdots,a_n)$. As a consequence, we prove the existence of triangle inscribed in $J$ similar to any given triangle.

\medskip
\noindent \textbf{Mathematics Subject Classifications (2020)}: \,53A04, \,51M04, \,30C35.
\end{abstract}
}

\begin{document}

\maketitle


\section{Introduction}

A polygon is an oriented piecewise linear closed curve in some Euclidean space $\mathbb R^k$ $(k\geq 2)$. The polygonal peg problem asks whether every simple closed curve admits an inscribed polygon of a given shape. An especially interesting type of such question is the longstanding square peg problem posed by Toeplitz~\cite{Toeplitz-1911}:
\emph{Does every continuous simple closed curve in the plane contain four points that are the vertices of a square?}
Toeplitz's problem has been solved affirmatively for curves under various regularity conditions by Emch~\cite{Emch-1913,Emch-1915,Emch-1916}, Schnirelman~\cite{Schnirelman-1944}, Jerrard~\cite{Jerrard-1961}, Stromquist~\cite{Stromquist-1989}, Tao~\cite{Tao-2017}, Matschke~\cite{Matschke-2022} and others. Yet the problem has remained open for general continuous curves. We refer the readers to Matschke's paper~\cite{Matschke-2014} for a survey together with an extensive list of further references on this problem.

Meanwhile, there was plenty of research on other types of quadrilaterals inscribed in given Jordan curves. Vaughan proved a celebrated result that every simple closed curve in the plane has an inscribed rectangle~\cite{Meyerson-1981}. Subsequently, additional notable progress with respect to edge ratios of rectangles inscribed in smooth Jordan curves appeared in the works of Matschke~\cite{Matschke-2014}, Hugelmeyer~\cite{Hugelmeyer-2021} and others. The latest development is that Greene-Lobb~\cite{Greene-2021} proved rectangles with arbitrary edge ratios can be inscribed in any smooth simple closed curve in the plane. In addition, Makeev~\cite{Makeev-1995}, Akopyan-Avvakumov~\cite{Akopyan-2018} and Matschke~\cite{Matschke-2021} studied the more general problem of finding inscribed cyclic quadrilaterals for some classes of simple closed curves, which was further extended by Greene-Lobb~\cite{Greene-2023} to including all smooth Jordan curves. For an account of further results, see the recent article of Schwartz~\cite{Schwartz-2022}. 

On the other hand, many interesting results were established on triangles inscribed in given Jordan curves. This is a bit easier to deal with than the cases of quadrilaterals, since the number of points to grapple with is only three. In fact, Meyerson~\cite{Meyerson-1980} and Kronheimer-Kronheimer~\cite{Kronheimer-Kronheimer-1981} proved the existence of triangles inscribed in any simple closed curve and similar to any given triangle. Moreover, Nielsen~\cite{Nielsen-1992} extended this result considerably by showing that the set of vertices of such triangles is dense in the given curve.

The above research was devoted to finding desired triangles and quadrilaterals up to similarities. For polygons with more than four edges, analogous results may not hold because the solution system is over-determined. For instance, it is easy to show an ellipse does not inscribe any cyclic $n$-gon with $n\geq5$.

A natural problem is to relax the requirement of polygon's angles. Precisely, one asks the following question: \emph{Does every simple closed curve have an inscribed polygon with any prescribed edge ratio? }When it is a circle, the problem was solved by the following theorem due to Penner~\cite{Penner-1987}. Mention that relevant results also appeared in works of Pinelis~\cite{Pinelis-2005}, Schlenker~\cite{Schlenker-2007} and Kou\v{r}imsk\'{a}-Skuppin-Springborn~\cite{Kour-Sku-Springborn-2016}. Let $|A_iA_j|$ be the Euclidean length of the line segment that connects $A_i$ and $A_j$. 

\begin{theorem}[Penner]\label{T-1-1}
Given $n$ $(n\geq 3)$ positive reals $a_1,\cdots,a_n$, each of which is less than the sum of the others, then there exists a convex polygon $\mathrm{Q}_n=A_0\cdots A_{n-1}$ inscribed in the unit circle  such that
the vector of distances
\[
\displaystyle{
|A_0A_1|,\,|A_1A_2|,\,\cdots,\,|A_{n-2}A_{n-1}|,\,|A_{n-1} A_0|
}
\]
is proportional to $(a_1,\cdots,a_n)$. Moreover, $\mathrm{Q}_n$ is unique up to isometries.
\end{theorem}

For other simple closed curves, to the best of our knowledge, few results have been obtained except that Milgram~\cite{Milgram-1943}, Meyerson~\cite{Meyerson-1980}, Wu~\cite{Wu-2004} and Makeev~\cite{Makeev-2005} investigated the existence of edge-regular polygons inscribed in certain smooth curves.

In this paper we establish the following result which extends the existence part of Penner's theorem.



\begin{theorem}\label{T-1-2}
Let $J$ be an oriented simple closed curve in ${\mathbb R}^k$ $(k\geq 2)$ that is differentiable with non-zero derivative at  $A_0\in J$. Let $a_1,\cdots,a_n$ $(n\geq 3)$ be positive reals, each of which is less than the sum of the others. Then there exists a tuple of ordered points $A_1,\cdots,A_{n-1}$ on $J$ after $A_0$ such that the vector of distances
\[
\displaystyle{
|A_0A_1|,\,|A_1A_2|,\,\cdots,\,|A_{n-2}A_{n-1}|,\,|A_{n-1} A_0|
}
\]
is proportional to $(a_1,\cdots,a_n)$.
\end{theorem}


Remark that this theorem includes the case that $J$ is a knotted curve. Meanwhile, taking $n=3$, we have the following consequence closely related to the works of Meyerson~\cite{Meyerson-1980}, Kronheimer-Kronheimer~\cite{Kronheimer-Kronheimer-1981},  Nielsen~\cite{Nielsen-1992} and Gupta-Salzedo~\cite{Gupta-2021}.

\begin{corollary}\label{C-1-5}
Let $J$ be an oriented simple closed curve in ${\mathbb R}^k$ $(k\geq 2)$ that is differentiable with non-zero derivative at a point $A_0\in J$. For any triangle $T$, there exist $A_1,A_2\in J$ such that  $T'=A_0A_1A_2$ forms a triangle similar to $T$.
\end{corollary}

\begin{figure}[htbp]
\centering
 \begin{minipage}[t]{0.48\textwidth}
  \centering
  \begin{tikzpicture}[scale=2.3]
    \draw (-0.9,0) -- (1,0) arc (0:180:1 and 1) -- cycle;
    \path  [draw=black,postaction={on each segment={mid arrow=black}},dashed]
     (0,1)--(-0.5,0)--(-0.08,0)--(0.62,0)--cycle;
    \draw [fill=black] (0,1) circle [radius=0.5pt] node [above] {$A_0$};
    \draw [fill=black] (-0.5,0) circle [radius=0.5pt] node [below] {$A_1$};
    \draw [fill=black] (-0.08,0) circle [radius=0.5pt] node [below] {$A_2$};
    \draw [fill=black] (0.62,0) circle [radius=0.5pt] node [below] {$A_3$};
    \end{tikzpicture}
 \end{minipage}
 \begin{minipage}[t]{0.48\textwidth}
  \centering
  \begin{tikzpicture}[font=\footnotesize,scale=2]
    \draw plot [smooth cycle] coordinates
    {(3.66,0.48)(3.4,0.6)(3.15,0.67)(2.95,0.68)(2.72,0.69)
    (2.48,0.62)(2.48,0.36)(2.6,0.11)(2.77,-0.02)(3.02,-0.09)(3.33,-0.12)(3.51,-0.15)(3.7,-0.19)
    (3.77,-0.28)(3.78,-0.43)(3.62,-0.55)(3.37,-0.54)(3.14,-0.49)(2.86,-0.42)(2.64,-0.43)(2.5,-0.53)
    (2.52,-0.65)(2.74,-0.83)(3.03,-0.96)(3.35,-1.01)(3.69,-0.97)(3.94,-0.84)(4.2,-0.6)(4.3,-0.34)(4.28,-0.08)(4.06,0.28)(3.83,0.4)};
    \path  [draw=black,postaction={on each segment={mid arrow=black}},dashed]
     (3.66,0.48)--(2.48,0.62)--(3.78,-0.43)--(2.52,-0.65)--cycle;
    \draw [fill=black] (3.66,0.48) circle [radius=0.4pt];
    \draw [fill=black] (2.48,0.62) circle [radius=0.4pt];
    \draw [fill=black] (3.78,-0.43)circle [radius=0.4pt];
    \draw [fill=black] (2.52,-0.65) circle [radius=0.4pt];
    \draw (3.75,0.47) node [above] {$A_0$};
    \draw (2.54,0.72) node [left] {$A_1$};
    \draw (3.75,-0.47) node [right] {$A_2$};
    \draw (2.6,-0.74) node [left] {$A_3$};
\end{tikzpicture}
 \end{minipage}
\caption{Inscribed polygons may have collinear vertices or cross points}
\label{fig2}
\end{figure}
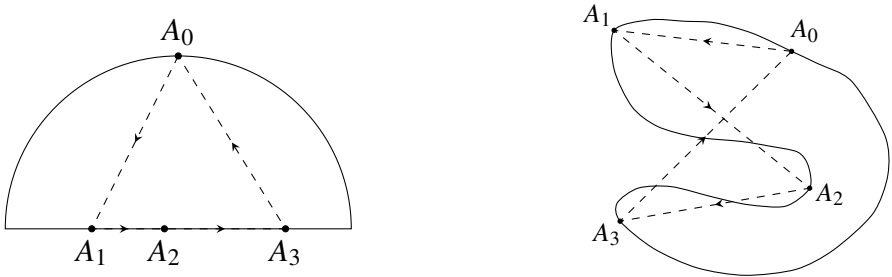

As in Figure~\ref{fig2}, an inscribed polygon $Q_n=A_0A_1\cdots A_{n-1}$ possibly presents complicated patterns. Fortunately, when $J$ is a strictly convex curve on the Euclidean plane $\mathbb R^2$, the polygon will be convex. Specifically, we have the following result.

\begin{theorem}\label{C-1-4}
The polygon $Q_n=A_0A_1\cdots A_{n-1}$ in Theorem~\ref{T-1-2} is convex provided  $J$ is a strictly convex curve in $\mathbb R^2$.
\end{theorem}

We hope to count the number of all ordered inscribed polygons starting at $A_0$ with a given edge ratio. Let
\[
\Delta^{n-1}=\{(a_1,\cdots,a_n)\in \mathbb{R}_+^n\,\big|\,a_1+\cdots+a_n=1\}
\]
and let $W$ consist of vectors $a=(a_1,\cdots,a_n)\in \Delta^{n-1}$ satisfying
\[
\displaystyle{
a_j<\sum_{i=1,i\neq j}^n a_i
}
\]
for $j=1,\cdots,n$. Apparently, $W$ is a non-empty open subset of $\Delta^{n-1}$. The following result provides a partial answer to this question.

\begin{theorem}\label{T-1-6}
Let $J$ be an oriented  simple closed $C^1$ curve in $\mathbb R^k$ $(k\geq2)$. For any $A_0\in J$ and for almost every $(a_1,\cdots,a_n) \in W$, there are at most finitely many ordered polygons inscribed in $J$ starting at $A_0$ with the vector of side-lengths proportional to $(a_1,\cdots,a_n)$.
\end{theorem}

The paper is organized as follows: In next section we prove a lemma concerning  polygonal curve inscribed in simple continuous arcs. We use the simplex $\Delta^{n-1}$ to parameterize all candidate inscribed polygonal curves and reduce the proof to determining the topological degree of the associated test map. In Section~\ref{S-3} we give proofs of the mains theorems. The key ingredient of this part is showing Theorem~\ref{T-1-2} by exhausting the given Jordan curve via a sequence of simple continuous arcs. Under the assumption that $J$ has a non-vanishing derivative at $A_0$, we assert that the resulting inscribed polygonal curve converges to a desired polygon. In Section~\ref{S-4}, we pose some questions for further developments.

Throughout this paper,  we use $o(\cdot)$ to indicate that the decay rate of a certain function or sequence is faster than that of another function or sequence.

\section{Polygonal curves inscribed in simple arcs}
We begin with a lemma regarding polygonal curves inscribed in simple arcs via the "configuration space/test map" scheme. Remark that similar methods have played significant roles in many problems in discrete geometry. See, for example, the works of Schnirelman~\cite{Schnirelman-1944}, Vre\'{c}ica-\v{Z}ivaljevi\'{c}~\cite{Vre-2011}, Matschke~\cite{Matschke-2011}, Akopyan-Karasev~\cite{Akopyan-2013}, Liu-Zhou~\cite{Liu-Zhou-2015}, Cantarella-Denne-McCleary~\cite{Cantarella-2022} and many others. 
\begin{lemma}\label{L-2-1}
Let $\gamma:[p,q]\to\mathbb R^k$ $(k\geq 2)$ be a simple continuous curve with $\gamma(p)\neq \gamma(q)$. Let $(l_1,\cdots,l_n)$ be a vector of positive reals. Then there exist
\[
p=s_0<s_1<\cdots<s_{n-1}<s_n=q
\]
such that the vector of distances
\[
|A_0A_1|,\,|A_1A_2|,\,\cdots,\,|A_{n-1}A_n|
\]
is proportional to $(l_1,\cdots,l_n)$, where $A_i=\gamma(s_i)$ for $i=0,1,\cdots,n$.
\end{lemma}

\begin{figure}[htbp]
\centering
\begin{tikzpicture}[scale=0.6]
\tikzstyle{every node}=[font=\small];
  \draw  [pattern color=lightgray] plot [smooth] coordinates {(7.79, 3.23)(7.18, 3.98)(6.55, 4.56)(5.38,5.26)(4.37,5.51)(3.35, 5.46)(2.25, 4.29)(1.35, 3.88)(0.35, 3.78)(-0.64, 3.9)(-1.39, 4.23)(-2.26, 4.78)(-3.07, 5.17)(-4.22, 5.41)(-5.13, 4.99)(-5.63, 4.33)(-6.03, 3.41)};
  \draw [fill=black] (7.79, 3.23) circle [radius=1.4pt] node [right] {$A_4$};
  \draw [fill=black] (4.37,5.51) circle [radius=1.4pt] node [above] {$A_3$};
  \draw [fill=black] (0.35, 3.78) circle [radius=1.4pt] node [above] {$A_2$};
  \draw [fill=black] (-2.26, 4.78) circle [radius=1.4pt] node [above] {$A_1$};
  \draw [fill=black] (-6.03, 3.41) circle [radius=1.4pt] node [left] {$A_0$};
  \draw [dashed] (7.79, 3.23)--(4.37,5.51)--(0.35, 3.78)--(-2.26, 4.78)--(-6.03, 3.41);
\end{tikzpicture}
\caption{Polygonal curve inscribed in a simple arc}
\label{figure}
\end{figure}
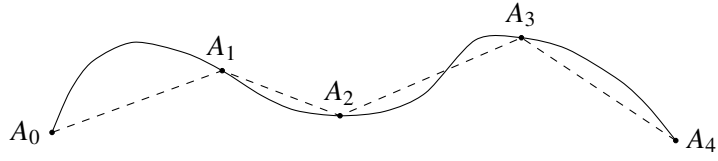

\begin{proof}
Recall the simplex
\[
\Delta^{n-1}=\{(\mu_1,\cdots,\mu_n)\,\big|\,\mu_1,\cdots,\mu_n>0,\,\mu_1+\cdots+\mu_n=1\}.
\]
For $\mu=(\mu_1,\cdots,\mu_n)\in\Delta^{n-1}$, let
\[
s_0=p,\,s_1=p+\sigma_1(q-p),\,\cdots,\,s_i=p+\sigma_i(q-p),\,\cdots,\,s_n=q,
\]
where
\[
\sigma_i=\mu_1+\cdots+\mu_i
\]
for $i=1,\cdots,n$. Setting $A_i=\gamma(s_i)$, we then define a test map $f_\gamma:\Delta^{n-1}\to\Delta^{n-1}$ as follows:
\[
f_\gamma(\mu)=\frac{(|A_0A_1|,\cdots,|A_{n-1}A_n|)}{\sum_{i=1}^{n}|A_{i-1}A_{i}|}.
\]
Note that
\[
\sum_{i=1}^{n}|A_{i-1}A_{i}|\geq |A_0A_n|=|\gamma(p)-\gamma(q)|>0.
\]
Hence $f_\gamma$ is well-defined. Meanwhile, the following Proposition~\ref{P-2-2} implies $f_\gamma$ is surjective. In particular, the vector
\[
\frac{(l_1,\cdots,l_n)}{l_1+\cdots+l_n}\in\Delta^{n-1}
\]
is in the image of $f_\gamma$, which completes the proof.
\end{proof}

\begin{proposition}\label{P-2-2}
$f_\gamma$ is a surjective map.
\end{proposition}
\begin{proof}
 Let $\{\mu^m\}\subset \Delta^{n-1}$ be a sequence approaching the faces. It is easy to see the image sequence $\{f_\gamma(\mu^m)\}$ also approaches the faces of $\Delta^{n-1}$. That means $f_\gamma$ is a proper map. Hence the topological degree of  $f_\gamma$ is well-defined. For $t\in[0,1]$, we define
\[
H(t,\mu)=t\mu+(1-t)f_\gamma(\mu).
\]
Note that $H$ forms a homotopy from $f_\gamma$ to the identity map. Moreover, for each fixed $t\in[0,1]$,  $H(t,\cdot)$  takes faces of $\Delta^{n-1}$ to faces, which implies the homotopy is proper. Due to the homotopy invariance property (see, e.g., ~\cite{Guillemin-2010,Hirsch-1976,Milnor-1997}), one obtains
\[
\deg(f_\gamma)=\deg(\mathrm{id})=1.
\]
As a result, $f_\gamma$ is surjective.
\end{proof}

\begin{remark}
An alternative approach to the above proposition is using the Knaster-Kuratowski-Mazurkiewicz lemma~\cite{Knaster-1929}. For this goal,  one needs to continuously extends $f_\gamma$ to map from $\overline{\Delta}_{n-1}$ to $\overline{\Delta}_{n-1}$ and construct a suitable KKM covering of  $\overline{\Delta}_{n-1}$.
\end{remark}

\section{Proofs of the main theorems}\label{S-3}
We hope to deduce Theorem~\ref{T-1-2} from Lemma~\ref{L-2-1}. To this end, we will take a sequence of simple arcs exhausting the given simple closed curve $J$ and then show the resulting sequence of inscribed polygonal curves converges to a non-degenerating polygon.

\begin{proof}[\textbf{Proof of Theorem~\ref{T-1-2}}]
Let $\gamma:[0,1]\to\mathbb R^k$ be a parametric representation of $J$ such that $A_0=\gamma(0)=\gamma(1)$ and $\gamma'(0)=\gamma'(1)\neq \vec{0}$. Now we apply Lemma~\ref{L-2-1} to the restriction of $\gamma$ to the segment $[p_m,q_m]\subset(0,1)$, where $p_m \to 0$ and $q_m\to 1$ as $m\to\infty$. We obtain
\[
p_m=s_0^m<s_1^m<\cdots<s_{n-1}^m<s_n^m=q_m
\]
such that the vector of distances
\[
|A^m_0A^m_1|,\,|A^m_1A^m_2|,\,\cdots,\,|A^m_{n-1}A^m_n|
\]
is proportional to $a=(a_1,\cdots,a_n)$, where $A^m_i=\gamma(s^m_i)$ for $i=0,1,\cdots,n$. After passing to a subsequence, we may assume that
\[
(s_0^m,s_1^m,\cdots,s_n^m)\to(s_0^\ast,s_1^\ast,\cdots,s_n^\ast).
\]

Let $A_i^\ast=\gamma(s_i^\ast)$ for $i=0,\cdots,n$. We claim $A_0^\ast A_1^\ast\cdots A_{n-1}^\ast$ is the desired polygon. Specifically, we need to check that $s_{i-1}^\ast<s_i^\ast$ for $i=1,2\cdots,n$. Otherwise, suppose there exists $i_0\in\{1,\cdots,n\}$ such that $s_{i_0-1}^\ast=s_{i_0}^\ast$. Then 
\[
|A^m_{i_0-1}A^m_{i_0}|\to0.
\]
Recall that the vector of distances
\[
|A^m_0A^m_1|,\,|A^m_1A^m_2|,\,\cdots,\,|A^m_{n-1}A^m_n|
\]
is proportional to $a=(a_1,\cdots,a_n)$. One obtains
\[
|A^m_{i-1}A^m_{i}|\to0
\]
for $i=1,\cdots,n$. Namely, the inscribed polygon  $A_0^\ast A_1^\ast\cdots A_{n-1}^\ast$ degenerates to the point $A_0=\gamma(0)=\gamma(1)$.

Without loss of generality, we now assume $s_{i}^\ast=0$ for $i<l_0$ and $s_i^\ast=1$ for $i\geq l_{0}$. By Taylor's formula, we get
\[
A_i^m=\gamma(s_i^m)=\gamma(0)+\gamma'(0)s_i^m+o(|s_i^m|)
\]
for $i=0,\cdots,l_0-1$. Hence
\[
|A_0A_0^m|=|\gamma'(0)|s_0^m+o(|s_0^m|).
\]
Moreover, for $i=1,\cdots,l_0-1$, we are ready to see
\[
|A_{i-1}^mA_{i}^m|=|\gamma'(0)|(s_i^m-s_{i-1}^m)+o(|s_i^m|).
\]
As a result,
\[
|A_0A_0^m|+\sum_{i=1}^{l_0-1}|A_{i-1}^mA_{i}^m|
=|\gamma'(0)|s_{l_0-1}^m+o(|s_{l_0-1}^m|).
\]
Similarly, one obtains
\[
\sum_{i=l_0+1}^{n}|A_{i-1}^mA_{i}^m|+|A_n^m A_0|=|\gamma'(1)|(1-s_{l_0}^m)+o(|1-s_{l_0}^m|).
\]
In addition,
\[
|A_{l_0-1}^mA_{l_0}^m|=|\gamma'(0)s_{l_0-1}^m+\gamma'(1)(1-s_{l_0}^m)|+o(|s_{l_0-1}^m|+|1-s_{l_0}^m|),
\]
Because $\gamma'(0)=\gamma'(1)\neq \vec{0}$, combining the above relations gives
\[
\lim_{m\to\infty}\bigg(\sum_{i=1,i\neq l_0}^n|A_{i-1}^mA_i^m|+|A_0A_0^m|+|A_n^m A_0|\bigg)\,\big/\,|A_{l_0-1}^mA_{l_0}^m|= 1.
\]
Note that
\[
\bigg(\sum_{i=1,i\neq l_0}^n|A_{i-1}^mA_i^m|\bigg)\,\big/\,|A_{l_0-1}^mA_{l_0}^m|=\bigg(\sum_{i=1,i\neq l_0}^n a_i\bigg)\,\big/\,a_{l_0}>1.
\]
It follows that
\[
\lim_{m\to\infty}\frac{|A_0A_0^m|+|A_n^m A_0|}{|A_{l_0-1}^mA_{l_0}^m|}<0,
\]
which leads to a contradiction. We thus finish the proof.

\end{proof}



In some respects, the inscribed polygon plays an analogous role to the weak solutions in PDE theory. Fortunately, under certain circumstances we are able to regularize the weak solutions to classical ones. 

\begin{proof}[\textbf{Proof of Theorem~\ref{C-1-4}}]
For simplicity, we identify the Euclidean plane $\mathbb R^2$ with the complex plane $\mathbb C$ and use $\operatorname{Arg}(\cdot)$ to denote the argument of a non-zero complex number.

Let $Q_n=A_0A_1\cdots A_{n-1}$ be an ordered polygon inscribed in $J$. Without loss of generality, we assume  the vertices $A_0,\cdots,A_{n-1}$  are arranged on $J$ in counter-clockwise order.
It suffices to check $Q_n$ is a convex polygon. Notice that $A_0,\cdots,A_{n-1}$ are distinct complex numbers. Let us choose the branch of the argument function $\operatorname{Arg}(\cdot)$ that ranges in $[0,2\pi)$. One needs to show
\[
\displaystyle{
\operatorname{Arg}\left(\frac{A_{i-1}-A_i}{A_{i+1}-A_i}\right)<\pi
}
\]
for $i=1,\cdots,n-1$. Indeed, when $J$ is strictly convex and $A_0,\cdots, A_{n-1}$ are orderly arranged on $J$, it is easy to see
\[
\displaystyle{
\operatorname{Arg}\left(\frac{A_{i-1}-A_i}{A_{i+1}-A_i}\right)=\alpha_i<\pi,
}
\]
where $\alpha_i$ is the inner angle of  triangle $A_{i-1}A_iA_{i+1}$ at $A_i$.
The theorem is proved.
\end{proof}

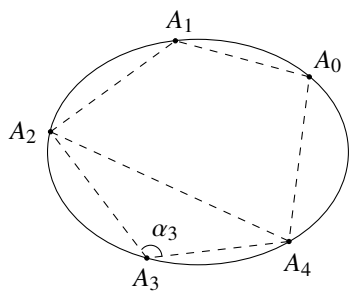
\begin{figure}[htbp]
\centering
\begin{tikzpicture}[font=\footnotesize]
\draw (0,-2) ellipse (2 and 1.5);
\draw (1.4,-0.8) node[right] {$A_0$};
\draw (-0.23,-0.55) node[above] {$A_1$};
\draw (-2.32,-2.03) node[above] {$A_2$};
\draw (-0.68,-3.4) node[below] {$A_3$};
\draw (1,-3.5) node[right] {$A_4$};
\draw [fill=black] (1.48,-1) circle [radius=0.8pt];
\draw [fill=black] (-0.3,-0.52) circle [radius=0.8pt];
\draw [fill=black] (-1.96,-1.73) circle [radius=0.8pt];
\draw [fill=black] (-0.68,-3.41) circle [radius=0.8pt];
\draw [fill=black] (1.21,-3.19) circle [radius=0.8pt];
\draw [dashed](1.48,-1)--(-0.3,-0.52)--(-1.96,-1.73)--(-0.68,-3.41)--(1.21,-3.19)--cycle;
\draw [dashed](-1.96,-1.73)--(1.22,-3.19);
\draw (-0.48,-3.39) arc (0:143:0.15);
\draw (-0.43,-3.32) node [above] {$\alpha_3$};
\end{tikzpicture}
\caption{Polygon inscribed in a strictly convex curve in the Euclidean plane}
\label{fig4}
\end{figure}

We now establish the finiteness property through the Regular Value Theorem and Sard's Theorem.

\begin{proof}[\textbf{Proof of Theorem~\ref{T-1-6}}] Without loss of generality, we assume $\gamma:[0,1]\to \mathbb R^k$ is a $C^1$ parametric representation of $J$ such that $\gamma(0)=\gamma(1)=A_0$. As before, for each $\mu=(\mu_1,\cdots,\mu_n)\in\Delta^{n-1}$, we set
\[
\sigma_i=\sum\limits_{j=1}^i\mu_j,\quad A_i=\gamma(\sigma_i)
\]
for $i=1,\cdots,n$ and define a test map $g_\gamma:\Delta^{n-1}\to W\subset \Delta^{n-1}$ as follows:
\[
g_\gamma(\mu)=\frac{(|A_0A_1|,\cdots,|A_{n-1}A_n|)}{\sum_{i=1}^{n}|A_{i-1}A_{i}|}.
\]

Since $\gamma$ is a $C^1$ map, it is easy to see $g_\gamma$ is continuous differentiable. Let $M_0$ denote the set of critical points of  $g_\gamma$. By Sard's Theorem (see, e.g., ~\cite{Guillemin-2010,Hirsch-1976,Milnor-1997}), the set
$W_0:=g_{\gamma}(M_0)$
has measure zero. Therefore, $W\setminus W_0$ is a subset of $W$ with full measure. Recall that
\[
\displaystyle{
\dim(\Delta^{n-1})=\dim(W)=n-1.
}
\]
For $a\in W\setminus W_0$, due to the Regular Value Theorem,  $g^{-1}_{\gamma}(a)$ is a discrete set. Meanwhile, a similar argument to Theorem~\ref{T-1-2} implies $g_\gamma$ is a proper map. That means $g^{-1}_{\gamma}(a)\subset \Delta^{n-1}$ is compact. Consequently, $g^{-1}_{\gamma}(a)$ is a finite point set, which completes the proof.
\end{proof}

\section{Some questions}\label{S-4}
\noindent
The paper leaves the following problems for further investigation.

\begin{itemize}

\item[$1.$] Does there exist a polygon inscribed in any simple closed curve with any prescribed edge ratio?  When $J$ is a plane curve and $n=3$, the problem was affirmatively solved by Meyerson~\cite{Meyerson-1980}, Kronheimer-Kronheimer~\cite{Kronheimer-Kronheimer-1981} and Nielsen~\cite{Nielsen-1992}. For $n>3$, one might attempt to approach the simple closed curve $J$ with a sequence of smooth curves $\{J^{(m)}\}$. Let $Q^{(m)}$ be a polygon inscribed in $J^{(m)}$ realizing the prescribed edge ratio. The limit of a convergent subsequence of $\{Q^{(m)}\}$ is then polygon $Q$ inscribed in $J$ with sides of lengths $\mu a_1,\cdots,\mu a_n$ for some $\mu \geq 0$. The difficulty is that $Q$ might degenerate in the sense that all of its vertices are at the same point of $J$.

\medskip
\item[$2.$] Under what conditions is the required polygon in Theorem~\ref{T-1-2} unique? Penner's theorem (Theorem~\ref{T-1-2}) asserts the uniqueness in case that $J$ is a circle. Moreover, for $C^1$ curves Theorem~\ref{T-1-6} indicates that generically the solution space is a finite point set. On the other hand, if $J$ includes the union of a circular arc and its center, then  $J$ has infinite inscribed triangles similar to some given shape. Apart from circles, we are looking for more simple closed curves so that the uniqueness still holds.

\end{itemize}

\section*{Acknowledgments}

Research of Y. Xu was supported by NSFC grant (No. 11631010) and the Postgraduate Scientific Research Innovation Project of Hunan Province (No. CX20210397). Research of Z. Zhou was supported by NSFC grant (No. 11631010).

\end{document}